\documentclass[10pt,a4paper, English]{article}
\usepackage[utf8]{inputenc}
\usepackage[english]{babel}
\usepackage[T1]{fontenc}
\usepackage{amsmath}
\usepackage{amsfonts}
\usepackage{amssymb}
\usepackage{makeidx}
\usepackage{graphicx}
\usepackage{lmodern}
\usepackage{mathrsfs}\newtheorem{Zae}{Zae}[section] 
\newtheorem{definition}[Zae]{Definition} 
\newtheorem{lemma}[Zae]{Lemma}
\newtheorem{prop}[Zae]{Proposition}
\newtheorem{theorem}[Zae]{Theorem}
\newtheorem{main}{THEOREM} 
\newtheorem{thm}[Zae]{Theorem} 
\newtheorem{coro}[Zae]{Corollary} 
\newtheorem{example}[Zae]{Example}
\newtheorem{remark}[Zae]{Remark}

\newtheorem{notation}[Zae]{Notation}

\newcommand{\NN}{\mathbb{N}}

\newcommand{\qed}{\raisebox{-.8ex}
{$\Box$}}
\newenvironment{proof}
{\noindent{\bf Proof.}}
{\hfill \qed  \\ \medskip }
\author{Matthias Gr\"uninger}
\title{Special Moufang sets of finite dimension}

\begin{document}
\maketitle
\begin{abstract} We prove that a special Moufang set with abelian root subgroups derives from a quadratic Jordan division algebra 
if a certain finiteness condition is satisfied. 
\end{abstract}
\section{Introduction}

The concept of a Moufang set was by introduced by J. Tits in \cite{T}. This  
concept is essentially equivalent to that of a group with 
split BN-pair of rank one. We first recall the definition of a Moufang 
set.
\begin{definition}\rm A {\it Moufang set} consists of a set $X$ with 
$|X|\geq 3$ and a family $(U_x)_{x\in X}$ of subgroups of $\mathrm{Sym}X$ 
such that 
\begin{itemize}
\item[\rm (MS1)] For all $x\in X$, the group $U_x$ fixes $x$ and acts 
regularly on $X\setminus \{x\}$.
\item[\rm (MS2)] For all $x,y\in X$ and $g\in U_y$ we have
$U_x^g =U_{xg}$. 
\end{itemize} 

\end{definition}

\noindent
The group $G^\dagger:=\langle U_x\mid x\in X\rangle$ is called 
the {\it little projective group} of the Moufang set. It acts 
$2$-transitively on $X$, and for all $x\in X$ the group 
$U_x$ is a normal subgroup of the stabiliser of $x$ in $G^\dagger$ acting 
regularly on $X\setminus \{x\}$. 

\begin{definition}\rm
A Moufang set $\big( X, (U_x)_{x\in X}\big)$ is called {\it proper} if its little projective group is not sharply $2$-transitive on $X$.
\end{definition}

\noindent
The prototype of a Moufang set is the projective line 
$\mathrm{PG}_1(\Bbbk)$ over a field $\Bbbk$. 
In this case we have $G^\dagger = \mathrm{PSL}_2(\Bbbk)$. As we will explain, one can generalise 
this by taking a quadratic Jordan division algebra instead of a field. 
  
\noindent  
As shown in \cite{DW}, 
every Moufang set can be constructed in the following manner: 
Let $(U,+)$ be a not necessarily abelian group with 
$|U|>1$, $X:=U\dot{\cup}\{\infty\}$ and  
$\tau \in \mathrm{Sym}X$ an element interchanging $0$ and $\infty$.
For $a\in U$ we set
$$\alpha_a:X\to X: x\mapsto \left\{ \begin{array}{ll} x+a, & \text{if } 
x\in U, \\ \infty, & \text{if } x=\infty. \end{array}\right.$$
Moreover, we set 
$U_\infty:=\{\alpha_a\mid a\in U\}$, $U_0:=U_\infty^\tau$,  
$U_a:=U_0^{\alpha_a}$ for all $a \in U$ and $\mathbb{M}(U,\tau):=
\big( X, (U_x)_{x\in X}\big)$. 

\noindent
This is not always a Moufang set. 
Condition (MS1) is straightforward, but 
(MS2) is not always satisfied. There is a nice criterion assuring that 
$\mathbb{M}(U,\tau)$ is a Moufang set. For $a\in U^\#:=U\setminus \{0\}$ 
set $\mu_a:= \alpha_a \alpha_{-a\tau^{-1}}^\tau \alpha_{-(-a\tau^{-1})
\tau}$ and $h_a:=\tau \mu_a$. We call $\mu_a$ the {\it $\mu$-map} and 
$h_a$ the {\it Hua map} associated to $a$. Note that $\mu_a$ interchanges 
$0$ and $\infty$ and thus $h_a$ fixes both elements. 

\begin{thm}\label{MS}{\rm (\cite[Theorem 3.2]{DW})} $\mathbb{M}(U,\tau)$ is a Moufang set if and only if 
$h_a\in \mathrm{Aut}(U) $ for all $a\in U^\#$, i.e. $(b+c)h_a =
bh_a +ch_a$ for all $a,b,c\in U$ with $a\ne 0$. If this holds, 
we have $\mathbb{M}(U,\tau) =\mathbb{M}(U,\mu_a)$ for all $a\in U^\#$.
\end{thm} 

\noindent
Note that if $\mathbb{M}(U,\tau)$ is a Moufang set, then 
$\mathbb{M}(U,\tau)=\mathbb{M}(U,\mu_a)$ for 
every $a\in U^\#$. 

\begin{definition}\rm Suppose that $\mathbb{M}(U,\tau)$ is a Moufang set. Then the group $H:=\langle \mu_a\mu_b\mid a,b\in U^\#\rangle$ is called the {\it Hua subgroup}
of $\mathbb{M}(U,\tau)$.\end{definition}

\noindent
One can show that $H:=G^\dagger_{0,\infty}$, so $\mathbb{M}(U,\tau)$ is proper if and only if $H\ne 1$.

\noindent
By Theorem \ref{MS} one can see that every quadratic Jordan division algebra 
gives rise to a Moufang set.
Recall the definition of a quadratic Jordan division algebra.

\begin{definition}\label{qj}
\rm Let $\Bbbk$ be a field, $J$ a $\Bbbk$-vector space, $Q: J\to \mathrm{End}_k(J)
:a\mapsto Q_a$ 
a map and $1\in J^\#$. 
Then $(J,Q,1)$ is called a {\it quadratic  
Jordan division algebra} if 
\begin{enumerate}
\item The map $Q$ is quadratic, i.e. $Q_{r\cdot a} =r^2 \cdot Q_a$ for 
all $a\in J$ and all $r\in \Bbbk$, and the map $(a,b)\mapsto Q_{a,b}:=
Q_{a+b}-Q_a-Q_b$ is $\Bbbk$-bilinear.
\item $Q_a$ is invertible for all $a\in J^\#$. We set $a^{-1}:=aQ_a^{-1}$ 
for $a\in J^\#$.
\item For all $a,b\in J$ we have 
\begin{itemize} 
\item[\rm (QJ1)] $Q_1=\mathrm{id}$.
\item[\rm (QJ2)] $Q_a V_{a,b}=V_{b,a} Q_a$, where $V_{a,b}\in 
\mathrm{End}_k(J)$ is defined by $cV_{a,b}:=bQ_{a,c}$ for all $c\in J$.
\item[\rm (QJ3)] $Q_{aQ_b} =Q_b Q_a Q_b$.
\end{itemize} 
 
\end{enumerate} 
\end{definition}

\begin{remark}\rm The concept of a quadratic Jordan algebra was introduced 
in \cite{M}. Normally, one requires that (QJ1-3) hold strictly, 
i.e. continue to hold in all scalar extensions of $J$. However, by the 
main theorem of  
\cite{G16}, this is not necessary for quadratic Jordan division algebras. 
\end{remark}

\noindent
The connection between Moufang sets and quadratic Jordan division algebras 
is established in the following theorem. 

\begin{thm}\label{MQ}{\rm (\cite[Theorem 4.1 and 4.2]{DW})} 
Let $(J,Q,1)$ be a quadratic Jordan division algebra over a field $k$. 
For $a\in J^\#$ we set $a\tau:=-a^{-1}$. 
Then 
$\mathbb{M}(J):=\mathbb{M}(J,\tau)$ is a Moufang set.
Moreover, we have $\tau =\mu_1$ and $h_a=Q_a$ for all $a\in J^\#$.
\end{thm}

\noindent
All examples of proper Moufang sets with abelian root groups that are known are isomorphic 
to $\mathbb{M}(J)$ for some quadratic Jordan division algebra $J$.
Therefore it is a major conjecture that every proper Moufang set with abelian root groups 
comes from a quadratic Jordan algebra.

\noindent
If $J$ is a quadratic Jordan division algebra, $\mathbb{M}(J,\tau)$ satisfies the identity  
$(-a)\tau =-a\tau$ for all $a\in J^\#$. 
Such Moufang sets are called {\it special}.

\begin{definition}\rm A Moufang set $\mathbb{M}(U,\tau)$ is called {\it 
special} if $(-a)\tau =-a\tau$ for all $a\in U^\#$.
\end{definition}

\noindent
By \cite{S} a Moufang set with abelian root groups is special or improper.
Moroever, the projective lines over $\mathbb{F}_2$ and $\mathbb{F}_3$ are up to isomorphism the only 
improper special Moufang sets. 
Therefore, for Moufang sets with abelian root groups, the properties special and proper are almost 
the same. 

\medskip\noindent
Given a proper Moufang set $\mathbb{M}(U,\tau)$ with $U$ abelian, the problem is to 
recover the Jordan structure. 
As a first step we can find a field over which the potential 
Jordan algebra 
is defined. By \cite[Proposition 4.6(5)]{DS}
either $U$ is torsion-free and uniquely divisible or $U$ is 
an elementary abelian $p$-group for some prime $p$. We say 
$\mathrm{char}U =0$ in the first case and $\mathrm{char}U=p$ in 
the second case. 
Thus $U$ is a vector space over $\Bbbk$, where $\Bbbk=\mathbb{Q}$ if 
char$U=0$ and $\Bbbk=\mathbb{F}_p$ if char$U=p$. This field 
$\Bbbk$ is called the {\it prime field} of $\mathbb{M}(U,\tau)$. 

\medskip \noindent
By Theorem \ref{MS} we may suppose that $\tau =\mu_e$ for some $e\in U^\#$.

\begin{notation}\rm Let $\mathbb{M}(U,\tau)$ be a proper Moufang set 
with $U$ abelian and $\tau=\mu_e$ for some fixed $e\in U^\#$. 
In the following, $h_0$ will denote the zero map of $U$, and 
we set $h_{a,b}:=h_{a+b}-h_a-h_b$ for all $a,b\in U$. 
\end{notation}

\noindent
Considering Theorem \ref{MQ}, the map $\mathscr{H}: U\to \mathrm{End}_\Bbbk(U):
a \mapsto h_a$ is a natural candidate for the quadratic 
map $Q$.
The map $\mathscr{H}$ depends on the choice of the neutral element $e\in U^\#$. If we choose another element $c\in U^\#$, then for $a\in U$ the Hua map $h_a^c$ is given by 
$h_a^c =h_c^{-1} h_a$. The triple $(U, \mathscr{H}^c,c)$ with $\mathscr{H}^c: U\to \mathrm{End}(U): a\mapsto h_a^c$ is called the $c$-{\it isotope} of $(U,\mathscr{H},e)$. 
The triple $(U,\mathscr{H},e)$ is a quadratic Jordan division algebra if and only if every $c$-isotope is a quadratic Jordan division algebra. 

\medskip\noindent
By \cite[Remark after 7.6.1]{DS2} all conditions for 
$(U, \mathscr{H},e)$ to be a 
quadratic Jordan division algebra but the biadditivity of $\mathscr{H}$ 
and Axiom (QJ2) are satisfied. 
In \cite{DS} it was proved that these conditions can be replaced by 
weaker conditions. In \cite{G22} the author of this paper proved 
that these conditions can be replaced by even weaker conditions. 
For the condition we are going to use in this paper we will need 
the following definition.

\begin{definition}\rm
Let $\mathbb{M}(U,\tau)$ be a special Moufang set with $U$ abelian and 
let $H$ be its Hua subgroup. We set $\mathscr{C}:= \{ T\in 
\mathrm{End}_H(U)\mid h_{aT}=T^2 h_a$ for all $a\in U\}$, the 
{\it centroid} of $\mathbb{M}(U,\tau)$. Moreover, we set 
$\mathscr{C}^*:=\mathscr{C}\cap \mathrm{GL}(U)$.
\end{definition}

\noindent
By \cite{DS}[Proposition 4.6(6)] $\mathscr{C}$ contains the prime field of $U$. 
Thus if char$U=0$, then $|\mathscr{C}|$ is infinite.

\noindent
We will make use of the following criterion:

\begin{theorem}\label{criterion} Let $\mathbb{M}(U,\tau)$ be a special Moufang set 
with $U$ abelian. Suppose that there is a subfield $\Bbbk$ of the centroid and an 
element $1\ne \lambda\in 
\Bbbk^*$ such  
that $h_{e,\lambda\cdot a} = \lambda\cdot h_{e,a}$ for all 
$a\in U^\#$. Then $(U,\mathscr{H},e)$ is a quadratic Jordan division 
algebra. 
\end{theorem}

\begin{proof}
By \cite[Theorem 2.6]{G22} it follows that the map $a\mapsto h_a$ is quadratic. The claim now follows from \cite[Theorem 4.6]{G22}.
\end{proof}

\noindent
The main result of this paper is Theorem \ref{main}:
\begin{main} Let $\mathbb{M}(U,\tau)$ be a special Moufang set 
with $U$ abelian and let $H$ be its Hua subgroup. 
Suppose that there is a subfield $\Bbbk$ of $\mathrm{End}_H(U)$ with 
\begin{itemize}
\item[\rm (i)] $\dim_\Bbbk U <\infty$, 
\item[\rm (ii)] $|\Bbbk\cap \mathscr{C}| =\infty$, 
\item[\rm (iii)] $U$ is generated by the set $\{eh_a \mid a\in U^\#\}$ as an $\mathrm{End}_H(U)$-module and
\item[\rm (iv)] there is $\lambda\in \Bbbk^* \cap \mathscr{C}$ such that $\lambda-1\in \Bbbk^* \cap \mathscr{C}$ as well.
\end{itemize}
Then $(U,\mathscr{H},e)$ is a quadratic Jordan division algebra. 
\end{main}
\noindent
Note that the second condition is automatically satisfied if char$U=0$ and the third  and fourth if char$U\ne 2$. In (iv), we can take $\lambda=-1$ if char$U\ne 2$. Thus we have

\begin{coro} Suppose that $\mathbb{M}(U,\tau)$ is a special Moufang set with $U$ abelian and torsion-free. If $U$ is finite-dimensional over a subfield $\Bbbk$ of the centroid of $\mathbb{M}(U,\tau)$, 
then $\mathbb{M}(U,\tau)$ arises from a quadratic Jordan division algebra.
\end{coro}

\noindent
As an application, we will discuss Moufang sets whose little projective 
group can be embedded into a sharply triply transitive group. 
A sharply triply transitive group whose point stabilisers have 
normal subgroup acting regularly arises from an algebraic structure called 
a {\it KT-nearfield}, named after Karzel and Tits. It is conjectured 
that every KT-nearfield is actually a commutative field with slightly 
modified multiplication. In \cite{Kerby91b} Kerby proved this for nearfields 
of characteristic different from $2$ having dimension $2$ over their 
kernel. We are able to strengthen this result:

\begin{main} Suppose that $F$ is a KT-nearfield with char$F\ne 2$ such 
that $F$ has finite dimension over its kernel. Then $F$ is a 
Dickson nearfield coupled to a commutative field 
which we also call $F$, and the 
corresponding sharply triply transitive group transitive is sandwiched between 
$\mathrm{PSL}_2(F)$ and $\mathrm{P\Gamma L}_2(F)$.
\end{main}

\noindent
For the proof of our results we will use an identity involving geometric series. 
However, in order to define an infinite series, we will need an appropriate 
topology which is a priori not given. In order to get such a topology, 
we will extend the original Moufang set by an ultrapower. Therefore, 
we will discuss ultraproducts at first. 

\bigskip
\noindent
{\bf Acknowledgement:} The author likes to thank Pierre-Emmanuel Caprace, Theo Grundh\"ofer and Bernhard M\"uhlherr for their valuable 
comments.

\section{Ultraproducts of Moufang sets}
In this section we will introduce ultraproducts of Moufang sets.
For an introduction on ultraproducts see \cite[\S 4]{CK}

\begin{definition}\rm Let $S$ be a non-empty set. 
\begin{enumerate}
\item 
A subset $\mathscr{F}$ of the power set $\mathcal{P}(S)$ of $S$ is called a {\it filter} on 
$S$ 
if \begin{enumerate}
\item $S\in \mathscr{F}$ and $\emptyset\not\in \mathscr{F}$.
\item For all $A,B\in \mathscr{F}$ we have $A\cap B\in \mathscr{F}$.
\item For all $A\in \mathscr{F}$ and all $A\subseteq B \subseteq S$ 
we have $B\in\mathscr{F}$.
\end{enumerate}
\item A filter $\mathscr{F}$ is called {\it fixed} if 
$\bigcap_{A\in \mathscr{F}} A\ne \emptyset$ and {\it free} otherwise.
\item A filter $\mathscr{F}$ is called an {\it ultrafilter} if 
for all $A\subseteq S$ either $A\in \mathscr{F}$ or $S\setminus A 
\in \mathscr{F}$.
\end{enumerate}
\end{definition}

\begin{example}\label{filter}\rm
\begin{enumerate}
\item If $\emptyset \ne B\subseteq \mathscr{F}$, then 
$\mathscr{F}_B:=\{A\subseteq S\mid B\subseteq A\}$ is a fixed 
filter on $S$. It is an ultrafilter if and only if $|B|=1$.
\item If $S$ is infinite, then $\mathscr{F}:=\{A\subseteq S\mid 
|S\setminus A| <\infty\}$ is a free filter on $S$, but not an 
ultrafilter.
\end{enumerate}
\end{example}

\begin{theorem}
Let $S$ be a non-empty set. 
\begin{enumerate}
\item If 
 $\mathscr{F}$ a filter on $S$, then there is an ultrafilter 
 $\mathscr{F}^* $ on $S$ 
 with $\mathscr{F}\subseteq \mathscr{F}^*$. 
\item If $S$ is infinite, then there is a free ultrafilter $\mathscr{F}$ 
on $S$.
\end{enumerate}
\end{theorem}

\begin{proof}
\begin{enumerate}
\item This follows by Zorn's Lemma, see for example
 \cite[Corollary 4.1.4]{CK} for details. 
\item This follows by (a) and Example \ref{filter}(b).
\end{enumerate}
\end{proof}

\noindent
Now let $S$ be a non-empty set and $\mathscr{F}$ an ultrafilter on $S$. Suppose that for every $s\in S$ there is a group/ring/etc. $X_s$. 
We set $$I_\mathscr{F}:=\{ (x_s)_{s\in S}\in \prod_{s\in S} X_s \mid \{ s\in S\mid x_s=0\}\in \mathscr{F}\}.$$ Then $I_\mathscr{F}$ is normal subgroup/ideal/etc. of $\prod_{s\in S} X_s$. 
We will write $\prod_{s\in S} X_s/\mathscr{F}$ instead of $\prod_{s\in S} X_s/I_\mathscr{F}$ and $[ (x_s)_{s\in S}]_\mathscr{F}$ or just $[(x_s)_{s\in S}] $ for 
$(x_s)_{s\in S}+I_\mathscr{F}$. Note that if $X_s\ne 0 $ for all $s\in S$ and $0\ne x\in \prod_{s\in S} X_s/\mathscr{F}$, then there is a representative $(x_s)_{s \in S}$ of $x$ such that 
$x_s \ne 0$ for all $s\in S$. We call $\prod_{s\in S} X_s/\mathscr{F}$ the {\it ultraproduct} of the $X_s$ with respect to $\mathscr{F}$. If $X_s=X$ for all $s\in S$, then 
$X^S/\mathscr{F}$ is called the {\it ultrapower} of $X$ with respect to $\mathscr{F}$. In this case, the map $x\mapsto [(x)_{s\in S}]$ is an embedding of $X$ into 
$X^S/\mathscr{F}$, thus we may consider $X$ as a substructure of $X^S/\mathscr{F}$.

\begin{lemma}\label{union} Let $\mathscr{F}$ be an ultrafilter on a set $S$ and 
let $A\in \mathscr{F}$. Suppose that $A=\dot{\bigcup}_{i=1}^n A_i$ 
for some $n\geq 1$. Then there is exactly one $1\leq i \leq n$ with 
$A_i\in \mathscr{F}$.
\end{lemma}

\begin{proof}
Induction on $n$. The claim is trivially true for $n=1$. 
For the inductional step suppose that $A=\dot{\bigcup}_{i=1}^{n+1} A_i$.
If $A_{n+1}\in \mathscr{F}$, then $A_i\not\in \mathscr{F}$ 
for $1\leq i \leq n$, since otherwise 
$\emptyset =A_i\cap A_{n+1}\in \mathscr{F}$. If $A_{n+1}\not\in 
\mathscr{F}$, then $S\setminus A_{n+1}\in \mathscr{F}$, thus 
$\dot{\bigcup}_{i=1}^n A_i= A \cap (S\setminus A_{n+1})\in \mathscr{F}$. 
The claim now follows by inductional hypothesis. 
\end{proof}

\begin{lemma}
Suppose that $\Bbbk_s$ is a field for all $s\in S$. Then $\prod_{s\in S} \Bbbk_s/\mathscr{F}$ is also field.
\end{lemma}

\begin{proof}
This is well-known, see for example \cite[Exercise 4.1.30]{CK}.
\end{proof}

\begin{lemma}\label{val} Let $S$ be a non-empty set, $\mathscr{F}$ an ultrafilter on $S$ 
and $\Bbbk$ a field. Set $\tilde{\Bbbk}:=\Bbbk^S/\mathscr{F}$. Then 
we have
\begin{enumerate}
\item $\Bbbk$ is algebraically closed in $\tilde{\Bbbk}$. 
\item Suppose that 
$f:S\to \Bbbk$ is a function that is not constant on any subset 
of $S$ contained in $\mathscr{F}$. Set $x:=f +\mathscr{F}$. 
Then there is a valuation 
$v:\tilde{\Bbbk}\to \mathbb{R}\cup \{\infty\}$ such that $v(x)=1$ 
and $v(a)=0$ for all $a\in \Bbbk^*$.
\end{enumerate}
\end{lemma}

\begin{proof}
\begin{enumerate}
\item Let $p\in \Bbbk[t]$ be a non-constant polynomial and 
$\alpha=[(\alpha_s)_{s\in S}]$ with $p(\alpha)=0$. 
Then $S':=\{ s\in S\mid p(\alpha_s) =0\} \in \mathscr{F}$. 
Let $\alpha_1,\ldots,\alpha_n$ be the roots of $p$ in 
$\Bbbk$. Set $S_i:=\{s\in S\mid \alpha_s=\alpha_i\}$ for $i=1,\ldots, n$.
Since $S'=\dot{\bigcup}_{i=1}^n S_i$, by \ref{union} 
there is exactly one $1\leq i \leq n$ 
with $S_i\in \mathscr{F}$. Thus $\alpha=\alpha_i\in 
\Bbbk$. Hence the claim follows.
\item By assumption $x\not\in \Bbbk$, thus $x$ is transcendental over 
$\Bbbk$ by (a). Let $\mathscr{B}$ be a transcendental base of 
$\tilde{\Bbbk}$ over $\Bbbk(x)$. Then $\mathscr{B}\cup \{x\}$ 
is a transcendental base of $\tilde{\Bbbk}$ over $\Bbbk$ and 
$\{x\}$ is a transcendental base of $\tilde{\Bbbk}$ over 
$\mathbb{K}:=\Bbbk(\mathscr{B})$. Thus there is a valuation 
$v$ of $\mathbb{K}(x)$ that is trivial on $\mathbb{K}\supseteq \Bbbk$ with 
$v(x) =1$. Since $\tilde{\Bbbk}$ is algebraic over $\mathbb{K}(x)$, 
we can extend $v$ to a valuation on $\tilde{\Bbbk}$ by \cite[XII 4.4]{L}. 
Thus the claim follows.
\end{enumerate}
\end{proof}

\begin{remark}\rm This valuation is not necessarily discrete. For example, 
if $S=\mathbb{N}$ and $f(n)=a^{n!}$ for some $a\in K^*$ of infinite order, 
then $x$ is a $n$-th power for all $n\in \mathbb{N}$, thus the value group 
contains $\mathbb{Q}$. The crucial thing is that the value group is 
Archimedean.\end{remark}

\begin{lemma}
Let $S$ be a non-empty set and $\mathscr{F}$ an ultrafilter on $S$. 
\begin{enumerate}
\item If $R_s$ is a ring and $M_s$ is an $R_s$-module for all $s\in S$, then $\prod_{s\in S} M_s/\mathscr{F}$ is a $\prod_{s\in S} R_s/\mathscr{F}$-module via 
$[(x_s)_{s\in S}] \cdot [(m_s)_{s\in S}] =[ (x_s\cdot m_s)_{s\in S}] $ for all $(x_s)_{s\in S}\in \prod_{s\in S} R_s$ and all $(m_s)_{s\in S}\in \prod_{s\in S} M_s$. 
\item Suppose that  $\Bbbk$ is a field, $V$ a $\Bbbk$-vector space of finite dimension and $B$ a $\Bbbk$-basis of $V$. Set $\tilde{\Bbbk}:=\Bbbk^S/\mathscr{F}$ and 
$\tilde{V}:=V^S/\mathscr{F}$. Then $B$ is $\tilde{\Bbbk}$-basis of $\tilde{V}$.
\end{enumerate}
\end{lemma}

\begin{proof}
\begin{enumerate}
\item This is obvious.
\item Since $B$ is finite, $V^S$ is generated by $B$ as a $\Bbbk^S$-module. 
One easily sees that $B$ is $\tilde{\Bbbk}$-linearly independent. Thus 
the claim follows.
\end{enumerate}
\end{proof}

\newpage
\noindent
Now suppose that $S$ is a non-empty set, $\mathscr{F}$ an ultrafilter on $S$ and $\mathbb{M}(U_s,\tau_s)$ a Moufang set for all $s\in S$. 
We suppose that $\tau_s =\mu_{a_s}$ for some $a_s\in U_s^\#$ for all 
$s\in S$. 
We set $U:=\prod_{s\in S} U_s/\mathscr{F}$. 
We define the map $\tau $ on $U^\#$ as follows: If $a\in U^\#$, choose a representative $(a_s)_{s\in S}$ of $a$ with $a_s\ne 0$ for all $s\in S$ and set $a\tau:=[(a_s\tau_s))_{s\in S}]$. 
This is well-defined. We get

\begin{theorem} For all $x=[(x_s)_{s\in S}], a=[(a_s)_{s\in S}]\in U$ with $a_s\ne 0$ for all $s\in S$ we have 
$xh_a =[(x_s h_{a_s})_{s\in S}]$. In particular $\mathbb{M}(U,\tau)$ is a Moufang set. Moreover, we have
\begin{enumerate}
\item The Moufang set $\mathbb{M}(U,\tau)$ is proper if and only if $\{ s\in S\mid \mathbb(U_s,\tau_s)$ is proper$\}\in \mathscr{F}$.
\item The Moufang set $\mathbb{M}(U,\tau)$ is special if and only if $\{ s\in S\mid \mathbb(U_s,\tau_s)$ is special$\}\in \mathscr{F}$.
\item If $\mathbb{M}(U_s,\tau_s)$ is special for all $s\in S$ and if 
$T_s \in \mathrm{End}(U_s)$ is in the centroid of $\mathbb{M}(U_s,\tau_s)$ for all $s\in S$, then 
$[ (T_s)_{s\in S}]$ is in the centroid of $\mathbb{M}(U,\tau)$. 
\end{enumerate}
\end{theorem}

\begin{proof}
Let $(x_s)_{s\in S} \in \prod_{s\in S} U_s$ and let $(a_s)_{s\in S}\in 
\prod_{s\in S} U_s^\#$. We have 
$S =S_1 \dot{\cup} S_2 \dot{\cup} S_3$ with 
$S_1:=\{ s\in S\mid x_s=0\}$, $S_2:=\{s\in S\mid x_s=a_s\tau_s^{-1}\}$ and 
$S_3:=\{s\in S\mid x_s\ne 0,a_s\tau_s^{-1}\}$. Then exactly one of 
the sets $S_1$, $S_2$ and $S_3$ are contained in $\mathscr{F}$ by \ref{union}.

\medskip
\noindent
Suppose that $S_1\in \mathscr{F}$. Then by definition $xh_a=0$ but 
also $\{ s\in S\mid x_sh_{a_s}=0\} =S_1\in \mathscr{F}$, thus the 
claim follows in this case. If $S_2\in \mathscr{F}$, then we have
$\{ s\in S\mid x\tau-a=0\}\in \mathscr{F}$, thus
$(x\tau-a)\tau^{-1}=\infty=(x\tau-a)\tau^{-1}-a\tau^{-1}$, 
hence $\big( (x\tau-a)\tau^{-1}-a\tau^{-1}\big)\tau =0$ and so 
$xh_a = \big( (x\tau-a)\tau^{-1}-a\tau^{-1}\big)\tau -(-a\tau^{-1})\tau 
=-(-a\tau^{-1})\tau $. However, for $s\in S_2$ we have 
$x_sh_{a_s} = a_s\tau^{-1}h_{a_s}= -(-a_s\tau_s^{-1})\tau_s$, thus 
the claim follows.

\medskip\noindent
Now for $s\in S_3$, we have $x_s,x_s\tau_s -a_s, 
(x_s\tau_s-a_s)\tau_s^{-1} -a\tau_s^{-1} \in U_s^\#$. Thus if $S_3\in 
\mathscr{F}$, then 
$(xh_a)_s =x_sh_{a_s}$ for all $s\in S_3$, hence the claim follows.

\medskip\noindent
Since every map $h_{a_s}$ is an automorphism of $U_s$ for all $s\in S$, 
it follows that $h_a$ is an automorphism of $U$. Thus $\mathbb{M}(U,\tau)$ 
is a Moufang set. 

\begin{enumerate}
\item For every $t\in \{ s\in S\mid \mathbb{M}(U_s,\tau_s)$ is proper$\}$ 
there are $x_t,a_t\in U_s^\#$ with $x_th_{a_t}\ne x_t$. 
Choose $x_s$ and $a_s$ arbitrarily for the other $s\in S$ and set 
$x:=[(x_s)_{s\in S}]$ and $a:=[(a_s)_{s\in S}]$. 
Then if $\{ s\in S\mid \mathbb{M}(U_s,\tau_s)$ is proper$\}\in \mathscr{F}$, 
we have $xh_a\ne x$, thus $\mathbb{M}(U,\tau)$ is proper. 

\medskip\noindent
Conversely, for $t\in \{s\in S\mid \mathbb{M}(U_s,\tau_s)$ is improper$\}$, 
we have $x_t h_{a_t}=x_t$ for all $x_t\in U$ and all $a_t\in U_t^\#$. 
Thus if this set is in $S$, the Moufang set $\mathbb{M}(U,\tau)$ is 
improper.
\item For all $t\in \{s\in S\mid \mathbb{M}(U_s,\tau_s)$ is special$\}$ 
we have $(-x_t)\tau_t =-x_t\tau_t$ for all $x_t\in U_t^\#$. 
Thus the claim follows easily.
\item This follows also by an easy calculation.
\end{enumerate}
\end{proof}
\newpage
\begin{remark}\rm 
It is conjectured that the root groups of a proper Moufang set are nilpotent, 
see \cite{DST}. If this conjecture is true, then there is a constant 
$c$ such that the root groups are nilpotent of class at most $c$. 
For if for every natural number $n$ there is a proper Moufang set 
$\mathbb{M}(U_n,\tau_n)$ such that $U_n$ is of class at least $n$ and 
if $\mathscr{F}$ is a free ultrafilter on $\mathbb{N}$, then 
$\prod_{n\in \mathbb{N}} \mathbb{M}(U_n,\tau_n)/\mathscr{F}$ is a proper 
Moufang set whose root groups are isomorphic to $\prod_{n\in \mathbb{N}} U_n/
\mathscr{F}$ and therefore not nilpotent. 
\end{remark}

\section{Special Moufang sets of finite dimension}

Before proving our main result we need some auxiliary results in algebra.

\begin{lemma}\label{Vandermonde} Let $\tilde{\Bbbk}$ be a field with valuation $v$, $\mathcal{O}$ its valuation ring, $\Bbbk$ a subfield of $\mathcal{O}$, 
$V$ a vector space over $\tilde{\Bbbk}$ and $\Lambda$ an $\mathcal{O}$-sublattice of $V$. Suppose that $v_1,\ldots,v_n  \in V$ and $\lambda_1,\ldots,\lambda_n\in \Bbbk$ are pairwise distinct
such that we have \\ $\sum_{j=1}^{n} \lambda_i^{j-1}\cdot v_j\in \Lambda$ for $1\leq i\leq n$. Then $v_1,\ldots,v_n\in \Lambda$. 
\end{lemma}

\begin{proof} The matrix $A:=(\lambda_i^{j-1})_{i,j=1,\ldots,n}$ is a Vandermonde matrix with entries in $\Bbbk$. Since the elements $\lambda_i$ are all distinct, we have $\det(A):=\prod_{i<j} (\lambda_j-\lambda_i)\in \Bbbk^* \subseteq \mathcal{O}^*$. Thus if $w_i:=\sum_{j=1}^{n} \lambda_i^{j-1} v_i$ and 
$A^{-1}=(b_{ij})_{i,j=1,\ldots,n} \in \mathrm{GL}_n(\mathcal{O})$, we have $v_i=\sum_{j=1}^n b_{ij} w_j\in \Lambda$. 
\end{proof} 

\begin{notation}\label{not}\rm
Let $\tilde{\Bbbk}$ be a field with valuation $v:
\tilde{\Bbbk}\to \mathbb{R}\cup \{\infty\}$, $\mathcal{O}$ its valuation ring, 
$V$ a $\tilde{\Bbbk}$-vector space and $B$ 
a basis of $V$. For $v=\sum_{b\in B} \lambda_b\cdot b\in V$ set $\omega(v):=\min\{ v(\lambda_b)\mid b\in B \}$. 
Then the map $\omega$ satisfies
\begin{itemize}
\item[\rm (i)] $\omega(v)=\infty \iff v=0$,
\item[\rm (ii)] $\omega(v+w) \geq \min\{ \omega(v), \omega(w)\}$,
\item[\rm (iii)] $\omega(\lambda\cdot v) =\omega(v)+v(\lambda)$
\end{itemize}
for all $v,w\in V$ and all $\lambda\in \tilde{\Bbbk}$. We will call 
$\omega$ the {\it minimum norm} with respect to $v$ and $B$. 
For all $c>0$ we get an ultrametric $d$ on $\tilde{U}$ by 
$d(v,w)=c^{\omega(v-w)}$ for $v,w\in \tilde{V}$.

\smallskip
\noindent
Moreover, we set 
$\Lambda_{r}:= 
\{ v\in \tilde{V}\mid \omega(v)\geq r\}$ for all $r\in\mathbb{R}$. 
Then $\Lambda_{r}$ is an $\mathcal{O}$-lattice.
\end{notation}

\begin{lemma}\label{isometry} Let $\tilde{\Bbbk}$, $V$ and $\mathcal{O}$ as in 
\ref{not}. 
Suppose that $\dim V=n<\infty$ and $B=(b_1,\ldots,b_n)$ is a $\Bbbk$-basis of $V$. 
Let $f:V\to V$ be a $\tilde{\Bbbk}$-linear map and $A:=(a_{ij})_{i,j=1,\ldots,n}$ 
its transformation matrix with respect to $B$. Then we have
\begin{enumerate}
\item The map $f$ is Lipschitz-continuous with respect to $\omega$ with Lipschitz-constant $L:=\min\{ v(a_{ij})\mid 1\leq i,j\leq n\}$, i.e. 
$\omega\big( f(v)\big) \geq L + \omega(v)$ for all $v\in V$.
\item Suppose that $\Bbbk$ is a subfield of $\mathcal{O}$. 
If $f$ is invertible and $a_{ij}\in \Bbbk$ for all $1\leq i,j\leq n$, then $f$ is an isometry with respect to $\omega$.
\end{enumerate}
\end{lemma}
\begin{proof}
For $v=\sum_{i=1}^n \lambda_i \cdot b_i\in V$ we have 
\begin{eqnarray*}
\omega\big( f(v)\big) &=& \omega\left( f\big( \sum_{i=1}^n \lambda_i \cdot b_i\big) \right) = \omega\big( \sum_{i=1}^n\lambda_i \cdot \sum_{j=1}^n a_{ij}\cdot b_j\big) 
\\
&=& \omega \big( \sum_{j=1}^n \sum_{i=1}^n \lambda_i \cdot a_{ij} \cdot b_j\big) \\
&=& \min\left\{ v\big( \sum_{i=1}^n \lambda_i \cdot a_{ij}\big) \mid 1\leq j\leq n\right\} \\
&\geq & \min \{ ( v( \lambda_i \cdot a_{ij}) \mid 1\leq i,j\leq n\} \\
&=& \min\{ v(\lambda_i) +v(a_{ij})\mid 1\leq i,j\leq n\} \\
&\geq & \min \{ v(\lambda_i)\mid 1\leq i\leq n\} +\min\{ v(a_{ij}) \mid 1\leq i,j\leq n\} \\
&=& \omega(v) +L.
\end{eqnarray*}
This proves (a). If $a_{ij}\in \Bbbk$ for all $1\leq i,j\leq n$ and $f$ is not the zero map, we have $L=0$, hence $\omega\big( f(v)\big) \geq 
\omega(v)$ for all $v\in V$. 
If $f$ is invertible, we also have $\omega\big( f^{-1}(v)\big) \geq \omega(v)$ for all $v\in V$. If we replace $v$ by $f(v)$, we get 
$\omega(v) \geq \omega\big( f(v)\big)$. Thus the claim follows. \end{proof}

\noindent
From now on we assume that $\mathbb{M}(U,\tau)$ is a special Moufang set with 
$U$ abelian and $\tau=\mu_e$ for some fixed element $e\in U^\#$.
Moreover, $\Bbbk$ is a subfield of $\mathrm{End}_H(U)$, where $H$ is the Hua subgroup of $\mathbb{M}(U,\tau)$.
\begin{notation}\rm For $a\in U$ and $n\in \NN$ we define $a^n$ recursively by 
$a^0:=e$, $a^1:=a$ and $a^{n+2}:=a^nh_a$. Moreover, for $a\in U^\#$ we set $a^{-1}:=-a\tau$.
\end{notation}

\begin{lemma}\label{ff}
For all $a\in U$ and $n,m\in\NN$ we have:
\begin{enumerate}
\item If $2m\geq n$, then $a^n =a^{n-2m} h_a^m$.
\item $h_{a^n}=h_a^n$.
\item $(a^m)^n =a^{nm}$.
\item $a^n h_{e,a}=2a^{n+1}$.
\item $(\lambda a)^n =\lambda^n a^n$ for all $\lambda\in \mathscr{C}$.
\end{enumerate}
\end{lemma}

\begin{proof} The first statement is very easy to see.
We prove (b), (c), (d) and (e) by induction on $n$. 
\begin{itemize}
\item[\rm (b)] For $n=0$ we have $h_{a^0}=h_e =1=h_a^0$, for $n=1$ we have 
$h_{a^1}=h_a =h_a^1$. Moreover, we have $h_{a^{n+2}}
=h_{a^n h_a} =h_a h_a^n h_a =h_a^{n+2}$. Thus the claim follows.
\item[\rm (c)] The statement is trivially true for $m\in \{0,1\}$, 
therefore we may assume $m\geq 2$. 
We have $(a^m)^0=e=a^{m\cdot 0}$ and 
$(a^m)^1= a^m =a^{m\cdot 1}$. Moreover, we have 
\begin{eqnarray*}
(a^m)^{n+2}&=&(a^m)^n h_{a^m}= a^{m\cdot n}h_a^m \\
&=&
a^{m\cdot n -2\cdot \lfloor \frac{m\cdot n}{2}\rfloor} h_a^{ \lfloor 
\frac{m\cdot n}{2}\rfloor} h_a^m \\
&=&
a^{m\cdot n -2\cdot \lfloor \frac{m\cdot n}{2}\rfloor} h_a^{ \lfloor 
\frac{m\cdot n}{2}\rfloor+m } \\
&
= & a^{m\cdot n +2m}\\ &=& a^{m\cdot (n+2)},
\end{eqnarray*} as claimed. 
\item[\rm (d)] For $n=0$ and $n=1$ this is \cite[Lemma 3.5 (a) and (c)]{G22}. By \cite[Proposition 5.8 (3)]{DS} the maps $h_a$ and $h_{e+a}$ commmute.
Hence $h_a$ commutes with $h_{a,e}=h_{a+e}-h_a-h_e = h_{a+e}-h_a -\mathrm{id}_U$. Thus the claim follows by induction. 
\item[\rm (e)] This is also clear for $n=0$ and $n=1$. 
Suppose the claim is true for $n$. 
Since $\lambda $ is 
contained in the centroid, we have 
$$(\lambda a)^{n+2} = (\lambda a)^n h_{\lambda a} = \lambda^n a^n 
h_{\lambda a} = \lambda^n \lambda^2 a^n h_a = \lambda^{n+2}a^{n+2}.$$
This proves the claim.
\end{itemize}
\end{proof}

\noindent
We will denote elements of the form $a^2 =eh_a$ as {\it squares}.

\begin{lemma}\label{squares}
If char$U\ne 2$, then the group $U$ is generated by squares.
\end{lemma}

\begin{proof}
We have $2a =eh_{a,e} = eh_{a+e}-eh_a-eh_e =(a+e)^2 -a^2 -e^2$ by \ref{ff}(d). Since $U$ is uniquely $2$-divisible, the 
claim follows.
\end{proof}

\begin{lemma}\label{f} For $a\in U^*$ and $b\in U$ we have $a^{-1}h_{a,b}=2b$.
\end{lemma}
\begin{proof}
By \cite[Lemma 5.7(2)]{DS} we have $a\tau h_{a,b}=-2b$. Thus the claim follows.
\end{proof}

\begin{lemma}\label{Lemma 1} For $a\ne e$ we have $(e-a)^{-1} h_{a,b}=\big( (e-a)^{-1}-e\big)h_{e,b}$.
\end{lemma}

\begin{proof}
For $a\ne 0$ we have $$a^{-1}h_{e-a,b} =a^{-1}h_{e,b}+a^{-1}h_{-a,b} =
a^{-1}h_{e,b} -2b =a^{-1}h_{e,b}-eh_{e,b}=( a^{-1}-e)h_{e,b}$$
by \cite[5.10(1)]{DS}, \ref{f} and \ref{ff}(d).
Replacing $a$ by $e-a$ yields
$$(e-a)^{-1}h_{a,b} =\big( (e-a)^{-1}-e\big) h_{e,b},$$
as desired.
\end{proof}

\begin{prop}\label{series} Suppose that $v:\Bbbk\to \mathbb{R}\cup \{\infty\}$ is a valuation and that $B$ is a basis of $U$. 
For $a=\sum_{b\in B} \lambda_b \cdot b\in U$ set $\omega(a):=\min\{
v(\lambda_b)\mid b\in B\}$. Suppose further that all Hua maps are continuous with respect to $\omega$. 
For all $t\in \Bbbk \cap \mathscr{C}$ with $v(t)>0$ and all $a\in U$ with $\omega(a^n)=0$ for all $n\in \mathbb{N}$ we have
$$(e-ta)^{-1} =\sum_{n=0}^\infty t^n a^n.$$
\end{prop}

\begin{proof}
For all $n\in \NN$ we have by \ref{ff}(d) and (e).
\begin{eqnarray*}
\big( (e-ta)^{-1}-\sum_{k=0}^n t^ka^k\big)h_{e-ta} &=& 
(e-ta)^{-1}h_{e-ta} -\sum_{k=0} t^k a^k h_{e-ta} \\ &=&
e-ta-\sum_{k=0}^n t^k a^k h_{e-ta} \\
&=& e-ta-\sum_{k=0}^n t^ka^kh_e -\sum_{k=0}^n t^ka^kh_{e,-ta}-\sum_{k=0}^n t^ka^k h_{-ta} \\ &=& 
e-ta-\sum_{k=0}^n t^ka^k -\sum_{k=0}^n (ta)^k h_{e,-ta}-\sum_{k=0}^n t^k 
t^2 a^k h_{a} \\
&=& -2ta-\sum_{k=2}^n t^k a^k +2\sum_{k=1}^{n+1} t^k a^{k} -\sum_{k=2}^{n+2} t^k a^k \\
&=& t^{n+1} a^{n+1}-t^{n+2} a^{n+2}.
\end{eqnarray*}

It follows that $\omega\left( \big( (e-ta)^{-1}-\sum_{k=0}^n t^ka^k\big)h_{e-ta}\right) \geq
(n+1)\cdot v(t)$. Thus the sequence 
$\left( \big( (e-ta)^{-1}-\sum_{k=0}^n t^ka^k\big)h_{e-ta}\right)_{n\geq 0}$ converges to 
$0$. Since $h_{e-ta}^{-1}= h_{(e-ta)^{-1}}$ is continuous, we also have 
$\lim_{n\to\infty} (e-ta)^{-1} -\sum_{k=0}^n t^k a^k =0$, thus $\sum_{n=0}^\infty t^n a^n
=(e-ta)^{-1}$.
\end{proof}

\noindent
We are now able to prove our main result:

\begin{theorem}\label{main} Suppose that there is a subfield $\Bbbk$ of $\mathrm{End}_H(U)$ with 
\begin{itemize}
\item[\rm (i)] $\dim_\Bbbk U <\infty$, 
\item[\rm (ii)] $|\Bbbk\cap \mathscr{C}| =\infty$, 
\item[\rm (iii)] $U$ is generated by the set of squares $\{a^2\mid a\in U\}$ as a $\mathrm{End}_H(U)$-module and
\item[\rm (iv)] there is $\lambda\in \Bbbk^* \cap \mathscr{C}$ such that $\lambda-1\in \Bbbk^* \cap \mathscr{C}$ as well.
\end{itemize}
Then $(U,\mathscr{H},e)$ is a quadratic Jordan division algebra. 
\end{theorem}

\begin{proof}
We embed $\mathbb{M}(U,\tau)$ in its ultrapower 
$\mathbb{M}(\tilde{U},\tilde{\tau})$, where 
$\tilde{U}:=U^{\mathbb{N}}/\mathscr{F}$ for some free ultrafilter 
$\mathscr{F}$ on $\mathbb{N}$. 
Moreover, let $\tilde{\Bbbk}:= \Bbbk^{\mathbb{N}}/\mathscr{F}$. 
Let $n\mapsto t_n$ be an injective 
map from $\mathbb{N}$ to $\Bbbk\cap \mathscr{C}$ and set $t:=[(t_n)_{n\in \mathbb{N}}] \in \tilde{\Bbbk}$. 
Then $t$ is contained in 
the centroid of $\mathbb{M}(\tilde{U},\tilde{\tau})$. By \ref{val}(b) we have a valuation 
$v:\tilde{\Bbbk}\to \mathbb{R}\cup \{\infty\}$ that is trivial on $\Bbbk$ 
with $v(t)=1$. 

\medskip\noindent
Let $B$ be a $\Bbbk$-basis of $U$. Then $B$ is also a $\Bbbk$-basis of 
$\tilde{U}$. The map $\omega$ and the lattices $\Lambda_r$ are 
defined as in Notation \ref{not}.

\bigskip
\noindent
Let $a,b\in U$. Using Lemma \ref{Lemma 1} for $ta$ in place of $a$ 
and Proposition \ref{series} we have
$$\big( \sum_{n=0}^\infty t^n a^n\big) h_{b,ta}=
(e-ta)^{-1} h_{b,ta} = \big( (e-ta)^{-1}-e\big) h_{b,e} 
 =\big( \sum_{n=1}^\infty t^na^n\big)h_{b,e}.$$
Since the maps $h_{b,ta}$ and $h_{e,b}$ are linear and thus continuous, we get
\begin{equation}\label{gs}
\sum_{n=0}^\infty t^n a^n h_{b,ta} =\sum_{n=1}^\infty t^na^nh_{b,e}.
\end{equation}
For all $n\geq 1$ we have $a^n h_{b,e} \in U$ and thus $\omega( t^n a^nh_{b,e}) \geq n$.
Since the series $\sum_{n=0}^\infty t^n a^n h_{b,ta}$ converges, the sequence 
$\big( t^n a^n h_{ta,b}\big)_{n\geq 0}$ converges to $0$. Thus there is $N\in \NN$ 
such that $\omega( t^n a^n h_{ta,b}) \geq 1$ for all $n>N$. Since
$\omega( \sum_{n=0}^\infty t^n a^n h_{ta,b}) =\omega(\sum_{n=1}^\infty t^n a^nh_{b,e})\geq \min\{ \omega(t^n a^n h_{b,e})\mid n\geq 1\}\geq 1$, 
we also have $\omega( \sum_{n=0}^N t^na^n h_{ta,b})\geq 1$. 
Therefore we have $\sum_{n=0}^N t^n a^n h_{ta,b}\in \Lambda_{1}$. 

\medskip \noindent
Now let $\lambda\in \mathscr{C} \cap \Bbbk$.
We have $t^n (\lambda a)^n h_{\lambda ta,\lambda tb} = t^n \lambda^{n+2} a^n h_{a,b}$. Replacing
$a$ by $\lambda a$ and $b$ by $\lambda b$ in \eqref{gs}, we get
$$\sum_{n=0}^\infty t^n \lambda^{n+2} a^n h_{ta,b} =\sum_{n=0}^\infty t^n (\lambda a)^n h_{\lambda t a,\lambda b}=\sum_{n=1}^\infty t^n \lambda^n 
a^n h_{\lambda b,e}.$$
Since $\omega( t^n \lambda^{n+2} a^n h_{ta,b}) =\omega(t^n a^n h_{ta,b})$, we conclude that
$\sum_{n=0}^N t^n \lambda^{n+2} a^n h_{ta,b}\in \Lambda_{1}$. Dividing by $\lambda^2$ yields
$\sum_{n=0}^N t^n \lambda^{n} a^n h_{ta,b}\in \Lambda_{1}$. Since $|k\cap \mathscr{C}| \geq N+1$, 
Lemma \ref{Vandermonde} yields that $t^n a^n h_{ta,b} \in \Lambda_{1}$ for all 
$n\geq 0$.

\medskip\noindent
Now let $m\geq 1$. 
Using the result above for $a^{2m+2}$ instead of $a$ and $b h_a^m$ 
instead of $b$ for $n=0$, we get 
$$ a^{2m} h_{t a^2 ,b} h_a^m =  eh_a^m h_{ta^2,b} h_a^m= 
 e h_{ta^2 h_a^m,b h_a^m}= e h_{t a^{m+2},bh_a^m} \in \Lambda_{1},$$
thus $(a^2)^m h_{ta^2,b} = 
a^{2m} h_{ta^2,b}\in \Lambda_{1}$ since $h_a^m$ is an 
isometry. Replacing $a$ by $a^2$ in \eqref{gs} yields 
$$ \sum_{n=0}^\infty t^n (a^2)^n h_{b,ta^2} =
\sum_{n=1}^\infty t^n (a^2)^n h_{b,e}.
$$ 
Since $\omega\big(t^n (a^2)^n h_{b,ta^2}\big) \geq 2$ for all $n\geq 1$ 
and $\omega\big( t^n (a^2)^n h_{b,e}\big) \geq 2$ for all $n\geq 2$, 
we get $eh_{ta^2,b} \equiv ta^2h_{e,b}$ mod $\Lambda_{2}$.

\medskip \noindent
Let $ \lambda\in \mathscr{C}\cap \Bbbk^*$ such that $\lambda-1\in \Bbbk^* \cap \mathscr{C}$. Replacing $t$ by $\lambda t$ and 
$b$ by $\lambda b$ we get
$$\lambda ta^2 h_{e,\lambda b} \equiv eh_{\lambda ta^2, \lambda b} 
\equiv \lambda^{2} eh_{ta^2,b} \equiv \lambda^2 ta^2 h_{e,b} \text{ mod }
\Lambda_{2}.$$
Dividing by $\lambda t$ yields
$$a^2 h_{e,\lambda b} \equiv \lambda a^2 h_{e,b} \text{ mod } \Lambda_{1}.$$
Since the map $x \mapsto x+\Lambda_{1}$ from $U$ to $\Lambda_{0}/\Lambda_{1}$ 
is injective, we $a^2 h_{e,\lambda b} =\lambda a^2 h_{e,b}$. Since $U$ is generated by 
the squares as a $\mathrm{End}_H(U)$-module, we get $ah_{e,\lambda b}=\lambda ah_{e,b}$.
Now let $0\ne c \in U$. The same argument for the $c$-isotope yields 
that $ah_c^{-1} h_{c,\lambda b}=\lambda ah_c^{-1} h_{c,b}$. Replacing $a$ by $ah_c$, 
we get that $ah_{\lambda b,c} =\lambda ah_{b,c}$ for all $a,b,c\in U$ with $c\ne 0$.
This equation is trivially also true for $c=0$. Thus the claim follows by
\ref{criterion}.
\end{proof}

\noindent
By the classification of quadratic Jordan division algebras in 
\cite[Theorem 15.7]{McZ}, a quadratic Jordan division algebra over a field of characteristic not $2$ is a skewfield, a Jordan algebra of Clifford type for an anisotropic 
quadratic form, an involutory set or an Albert divison algebra. The corresponding little projective group is a $\mathrm{PSL}_2$ over a skewfield, an orthogonal group, a unitary group or an exceptional algebraic 
group of type $\mathrm{E}^{78}_{7,1}$. Therefore, we can formulate our result in the language of group theory. First, we need two auxiliary results.

\begin{lemma}\label{2trans} Let $G$ be a doubly transitive permutation group on a set $X$ such that for some $y\in X$ there is a normal subgroup $U_y$ of $G_y$ acting regularly on $X\setminus \{y\}$. For every $x\in X$ choose $g_x\in G$ 
with $yg_x =x$ and define $U_x:=U_y^{g_x}$. Then $U_x$ does not depend of 
the choice of $g_x$, and $\big( X, (U_x)_{x\in X}\big)$ is a Moufang set 
whose little projective group $G^\dagger$ is normal in $G$.
\end{lemma}

\begin{proof}
If $h_x$ is another element with $yh_x=x$, then $h_x g_x^{-1} \in G_y$, 
so $U_y^{h_x g_x^{-1}} =U_y$, hence $U_y^{h_x}=U_y^{g_x}$. This proves the 
first part. For $x,z\in X$ and $g\in U_x$ we therefore have 
$U_z^g =U_y^{g_yg}=U_{yg_yg} =U_{zg}$. This proves that 
$\big( X, (U_x)_{x\in X}\big)$ is a Moufang set. Moreover, 
$\{U_x\mid x\in X\}$ is conjugacy class of subgroups in $G$, hence 
$G^\dagger=\langle U_x\mid x\in X\rangle$ is normal in $G$.
\end{proof}

\begin{lemma}\label{centroid} Let $\mathbb{M}(U,\tau)$ be a proper Moufang set with $U$ abelian. Let $G^\dagger$ be its little projective group. Suppose that $G^\dagger \trianglelefteq G \leq \mathrm{Sym}(X)$. 
Then 
$Z(G_{\infty,0})$ is contained in the centroid of $\mathbb{M}(U,\tau)$.
\end{lemma}
\begin{proof} Let $Z:=Z(G_{\infty,0})$ and $H$ be the Hua subgroup of $\mathbb{M}(U;\tau)$. Then $Z\subseteq \mathrm{End}_H(U_\infty)$ since 
$H\leq G_{\infty,0}$, For every $a\in U^\#$, the element $\mu_a $ is contained in 
$G_{\{\infty,0\}}$ and therefore normalises $Z$. Since $\mu_a^2=1$, we have 
$h^{\mu_a}=h^{-1}$ for all $h\in Z$ or there is $1\ne z\in Z$ with $z^{\mu_a}=z$. 
Suppose the latter case holds. Then for all $b\in U^\#$ we have 
$z=z^{\mu_a\mu_b}= z^{\mu_b}$ since $\mu_a\mu_b\in H\leq G_{\infty,0}$. Therefore 
we have $\mu_b=\mu_b^h =\mu_{bh}$ for all $b\in U^\#$. Thus we have $bh\in \{b,-b\}$ 
for all $b\in U$. Therefore $h=1$ or $b^h =-b$ for all $b\in U$. In both cases we have
$h^{-1}=h=h^{\mu_a}$.

\medskip\noindent
Thus we have $h^{\mu_a}=h^{-1}$ for all $h\in Z$ and all $a\in U^\#$. We get 
$$h_{ah}= \mu_e \mu_{ah}= \mu_e \mu_a^h =\mu_e h^{-1} \mu_a h =h \mu_e\mu_a h =
h h_a h = h^2 h_a$$
\end{proof}

\begin{coro}\label{coro0} Let $G$ be a group acting doubly transitively on a set $X$. Suppose that for some $x\in X$ there is an abelian subgroup $U_x\leq G_x$ with
\begin{itemize}
\item[\rm (i)] $U_x$ acts regularly on 
 $X\setminus \{x\}$,
\item[\rm (ii)] $U_x$ does not contain an involution,
\item[\rm (iii)] $U_x$ is finitely generated as a $Z(G_{x,y})$-module for all 
$y\in X\setminus \{x\}$. 
\end{itemize}
Let $G^\dagger$ be the subgroup of $G$ generated by $\{U_x^g\mid g\in G\}$. 
Then $G^\dagger\trianglelefteq G$ and there is a commutative field $\Bbbk$ such 
that one of the following holds:
\begin{itemize}
\item[\rm (I)] char$\Bbbk=2$ and $G^\dagger\cong \mathrm{Aff}(\Bbbk)$. 
\item[\rm (II)] char$\Bbbk\ne 2$ and either $G^\dagger$ is either a classical group of relative $\Bbbk$-rank one or an exceptional algebraic group of type 
$\mathrm{E}^{78}_{7,1}$. 
\end{itemize}
\end{coro}

\begin{proof} 
By \ref{2trans} $\mathbb{M}:=\big( X, (U_x)_{x\in X}\big)$ is a Moufang set and $G^\dagger$ is normal in $G$. Suppose that $\mathbb{M}$ is improper. Hence $G^\dagger$ is sharply $2$-transitive. Since $U_x$ is abelian, by \cite{Tits}[Remark 2 on page 
47/48] or \cite{DM}[Theorem 3.4B]
there is a commutative field $\Bbbk$ such that $G^\dagger \cong \mathrm{Aff}(\Bbbk)$. Since $\Bbbk^*\cong U_x$ does not contain an involution, we have char$\Bbbk=2$. Hence we are in case (I).

\medskip\noindent
Now suppose that $\mathbb{M}$ is proper. Since $U_x$ is abelian, the Moufang set $\mathbb{M}$ is special by \cite{S}. By \cite{SW} the group $G_{x,y}$ acts irreducibly on $U_x$, thus 
$K:=\mathrm{End}_{G_{x,y}}(U_x)$ is a skewfield by Schur's Lemma. Let $\Bbbk$ be the subfield of $K$ generated  by $Z(G_{x,y})$. Then $\Bbbk$ is commutative, and since there is no $2$-torsion in $U_x$, we have 
char$\Bbbk\ne 2$. Suppose that $Z(G_{x,y})$ is finite. Then $Z(G_{x,y})$ is algebraic over the prime field of $\Bbbk$, and so $\Bbbk$ is finite-dimensional 
over its prime field. Moreover, we have char$\Bbbk>0$, since the map $n^2 \cdot \mathrm{id}_U$ is contained in $H$ for all $n\in \mathbb{N}$ with $n$ relatively prime to the characteristic of $U$. Therefore $\Bbbk$ is finite as well. Since $U_x$ is finitely generated as a $Z(G_{x,y})$-module, 
we have $\dim_\Bbbk U <\infty$, hence $U$ is finite. By \cite{S1} we get $G^\dagger \cong \mathrm{PSL}_2(F)$ for some finite field $F$ with $\Bbbk \subseteq F$.

\medskip\noindent
Therefore we may assume that $Z(G_{x,y})$ is infinite. By \ref{centroid} and our main theorem \ref{main} we conclude $\mathbb{M}(U,\tau)\cong \mathbb{M}(J)$ 
for some quadratic Jordan division algebra $J$. If $J$ is special, then $G$ is a classical group by 
\cite{Tim2}[Corollary 5.4]. If $J$ is exceptional, then $G^\dagger$ is an 
exceptional algebraic group of type $\mathrm{E}^{78}_{7,1}$ by \cite[Remark after 2.29]{DMB}.
\end{proof}

\begin{definition}\rm
\begin{enumerate}
\item 
Let $G$ be a group acting on a set $X$. Then $G$ is called {\it Zassenhaus transitive} on $X$ if
\begin{enumerate}
\item $G$ acts doubly transitively on $X$, 
\item $G_{x,y}\ne 1$ but $G_{x,y,z}=1$ for all pairwise distinct $x,y,z\in X$,
\item for all $x\in X$ there is a normal subgroup $U_x$ of $G_x$ acting regularly on $X\setminus \{x\}$ 
and 
\item there is no regular normal subgroup $N$ of $G$.
\end{enumerate}
\item Let $\mathbb{M}=\big( X, (U_x)_{x\in X}\big)$ be a Moufang set with 
little projective group $G^\dagger$. We call $\mathbb{M}$ a {\it Zassenhaus 
Moufang set} if $G^\dagger$ is Zassenhaus transitive on $X$. 
\end{enumerate}
\end{definition}

\noindent Note that if $\mathbb{M}(U,\tau)$ is a Moufang set with Hua subgroup 
$H$, then $\mathbb{M}(U,\tau)$ is Zassenhaus if and only if $H\ne 1$ and 
$H$ acts freely on $U$. If $\mathbb{M}=\mathbb{M}(J)$ for some quadratic Jordan 
division algebra $(J,Q,e)$, this means that if $x,x_1,\ldots,x_n\in J^\#$ with 
$xQ_{x_1}\ldots Q_{x_n}=x$, then $yQ_{x_1}\ldots Q_{x_n}$ for all $y\in J$.

\bigskip
\noindent The following statement is a special case of \cite[Theorem 4.8]{ACP}.
However, unlike this theorem it does not use the classification of quadratic Jordan 
division algebras.
\begin{theorem}\label{abelian}
Let $(J,Q,e)$ be a quadratic Jordan division algebra such that 
$\mathbb{M}(J)$ is a Zassenhaus Moufang set. Then $Q_aQ_b=Q_bQ_a$ for all 
$a,b\in J$.
\end{theorem}

\begin{proof}
We will make use of the identities in \cite[Chapter 1, Section 2 and 3]{J}.
Note that $V_{a,b}$ in our notation corresponds to $V_{b,a}$ in 
Jacobson's notation. As Jacobson we set $V_a:=V_{a,e}=V_{e,a}=Q_{e,a}$ 
for $a\in J$. 

Let $a,b\in J^\#$.
By identity QJ4' on \cite[p. 18]{J} we have 
$$V_{b,a}Q_a= Q_aV_{a,b}= Q_{a,bQ_a}.$$ 
Applying $Q_a^{-1}$ on the left and right and using (QJ3) we get
$$Q_a^{-1}V_{b,a}=V_{a,b}Q_a^{-1}=Q_a^{-1}Q_{a,bQ_a}Q_a^{-1}=
Q_{a^{-1},b}.$$
Thus 
$$eQ_{a,b}Q_a^{-1}=eV_{a,b}Q_a^{-1}=eQ_{a^{-1},b}.$$
Therefore we get
$$eQ_{a,b}Q_a^{-1}Q_b^{-1}=eQ_{a^{-1},b^{-1}}=eQ_{a,b}Q_b^{-1}Q_a^{-1}.$$
If $eQ_{a,b}\ne 0$, this implies $Q_aQ_b =Q_b Q_a$.

Now suppose that $eQ_{a,b}=0$. We will show that $Q_aQ_b =Q_bQ_a$ is also true in this case. 
Since $eQ_{a,b}=0=eQ_{a,-b}$, 
we have $eQ_{a+b}=eQ_a+eQ_b =eQ_{a-b}$, thus $Q_{a+b}=Q_{a-b}$ and 
therefore $2Q_{a,b}=0$. Moreover, we have 
$$a^{-1}Q_{a,b}=aV_{b,a^{-1}}=a^{-1}Q_{a^{-1}}^{-1}V_{b,a^{-1}}=
a^{-1} V_{a^{-1},b} Q_a =bQ_{a^{-1},a^{-1}}Q_a=2bQ_aQ_a^{-1}=2b,$$
therefore we only have to deal with characteristic $2$. 

By identity QJ18 of \cite{J} we have
$aQ_bV_a +aV_bV_{b,a}=aV_{a,b}V_b+aV_aQ_b$. 
Since $aV_b=0$, $aV_a=eQ_{a,a}=2eQ_a=0$ and 
$aV_{a,b}=bQ_{a,a}=2bQ_a=0$ we have $eQ_{aQ_b,a}=aQ_bV_a=0$.
By identity QJ30 of \cite{J} we have
$$0=eQ_{eQ_{a,b}} =eQ_aQ_b+eQ_bQ_a +eQ_{a,aQ_b},$$
thus 
$eQ_aQ_b=eQ_bQ_a$. Thus the claim follows.
\end{proof}

\begin{coro}\label{coro1} If $\mathbb{M}(J)$ is a Zassenhaus Moufang set, 
then $J$ is a commutative field 
or there is a commutative field $\Bbbk$ with 
$\Bbbk^2 \subseteq J\subseteq \Bbbk$.
\end{coro}

\begin{proof}
This follows by the main theorem of \cite{G10}.
\end{proof}

\begin{coro}\label{coro2} Suppose that $G$ acts Zassenhaus transitively on a set $X$ and that for all $x,y\in X$ pairwise distinct 
\begin{itemize}
\item[\rm (i)] the group $U_x$ is abelian, but not elementary-abelian of exponent $2$ and 
\item[\rm (ii)] $U_x$ is finitely generated as a $Z(G_{x,y})$-module.
\end{itemize}
Then there is a field $F$ with char$F\ne 2$ and a group $\mathrm{PSL}_2(F)\leq 
\mathscr{G}\leq \mathrm{P\Gamma L}_2(F)$ such that $(X,G)$ and $\big( 
\mathrm{PG}_1(F), \mathscr{G}\big)$ 
are isomorphic as permutation groups, where $\mathrm{PG}_1(F)$ denotes the 
projective line over $F$.
\end{coro}

\begin{proof} Since $G$ is Zassenhaus transitive on $X$, there is no regular 
normal subgroup of $G$, so case (I) of \ref{coro0} can be excluded. By 
\ref{coro1} there is a commutative field $F$ with char$F\ne 2$ such that the corresponding 
Moufang set is isomorphic to the projective line over $F$. The little projective 
group $G^\dagger$ is isomorphic to $\mathrm{PSL}_2(F)$. Thus $G$ is isomorphic 
to a subgroup of $\mathrm{Aut}\big(\mathrm{PSL}_2(F)\big)$ which is $\mathrm{P\Gamma L}_2(F)$ by \cite{D}[IV, \S 6].
\end{proof}

\begin{remark}\rm By \cite{Maeu}, the group $\mathscr{G}$ is contained in $\mathrm{PGL}_2(F)$ if and only $G_{xy}$ is abelian. 
\end{remark}
\section{KT-nearfields}

\noindent
KT-nearfields are the algebraic structures for Moufang sets whose 
little projective group is contained in a sharply $3$-transitive 
groups. At first, we recall some notion about nearfields. For an introduction 
to nearfields see \cite{W}.

\begin{definition}\rm
Let $F$ be a set with two binary operations $+$ and $\cdot$ and $0,1\in F$ two distinct elements. 
Then $(F,+,\cdot,0,1)$ is called a {\it (right) nearfield} if 
\begin{enumerate}
\item $(F,+,0)$ is a group.
\item $(F^*,\cdot,1)$ is a group, where $F^*:=F\setminus \{0\}$.
\item $(a+b)\cdot c = a\cdot c+b\cdot c$ for all $a,b,c\in F$.
\item $a\cdot 0=0$ for all $a\in F$.
\end{enumerate}
\end{definition}

\noindent
It can be seen that the additive group of a nearfield is necessarily abelian, see \cite{W}[I, \S 2, (2.3)].

\begin{definition}\rm
Let $F$ be a nearfield.
\begin{enumerate}
\item $Z(F):=\{a\in F\mid a\cdot b=b\cdot a$ for all $b\in F\}$ is called the {\it centre} of $F$.
\item $\Bbbk(F):=\{ a\in F\mid a\cdot (b+c) =a\cdot b +a\cdot c$ for all $b,c\in F\}$ is called the {\it kernel} of $F$.
\end{enumerate}
\end{definition}

\noindent
One easily sees that $\Bbbk(F)$ is a skewfield containing $Z(F)$ and that $F$ is a left vector space over $\Bbbk(F)$.

\begin{example}\label{dickson}\rm
\begin{enumerate}
\item Let $F$ be a skewfield and $\varphi: F^*\to \mathrm{Aut}(F)$ a map with 
$\varphi( a^{\varphi(b)} b) =\varphi(a)\varphi(b)$ for all $a,b\in F^*$. Such a map is called a {\it coupling}. 
We define a new multiplication $\cdot $ on $F$ by 
$$a\cdot b :=\left\{ \begin{array}{ll} a^{\varphi(b)}b, & \text{if } b\in F^*,\\
0, & \text{if } b=0\end{array}\right.$$ for all $a,b\in F$. Then $F^\varphi:=(F,+,\cdot,0,1)$ is nearfield, called a {\it Dickson nearfield}. We say that 
$F^\varphi$ is {\it coupled to} $F$.
\item A nearfield that is not a Dickson nearfield is called {\it wild}. There are up to isomorphism exactly seven finite wild 
nearfields. Infinite wild nearfields were constructed in \cite{ZG}, \cite{Gru}, \cite{GG} and \cite{Gr}.
\end{enumerate}
\end{example}

\begin{definition}\rm Let $F$ be a nearfield. For $a\in F^*$ and $b\in F$ let $t_{a,b}:F\to F: x\mapsto ax+b$.
Set $\mathcal{T}(F):=\{t_{1,b} \mid b \in F\}$ and 
$\mathrm{Aff}(F):= \{t_{a,b} \mid a\in F^*,b\in F\}$, the {\it affine group} of $F$.
\end{definition}

\noindent
We have
\begin{theorem}
\begin{enumerate}
\item 
$\mathrm{Aff}(F)$ is a subgroup of $\mathrm{Sym}(F)$ that acts sharply $2$-transitively on $F$.
\item $\mathcal{T}(F)$ is a normal subgroup of $\mathrm{Aff}(F)$ that acts regularly on $F$. Moreover, we have 
$\mathcal{T}(F)\cong (F,+)$. \end{enumerate}
\end{theorem}

\noindent
On the converse, we have 
\begin{theorem} Let $X$ be set, $G\leq \mathrm{Sym}(X)$ a group acting sharply $2$-transitively on $X$ and $N$ a normal subgroup 
of $G$ acting regularly on $X$. 
Then there is a nearfield $F$, such that $(G,X)$ and $\big( \mathrm{Aff}(F),F\big)$ are isomorphic as permutation groups. Moreover, 
if $\varphi:G\to \mathrm{Aff}(F)$ is an isomorphism, then $\varphi(N)=\mathcal{T}(F)$.
\end{theorem}

\noindent Proofs of these well-known results can 
for instance be found in \cite{KW2}.

\medskip\noindent
The question whether the sharply $2$-transitive action of $\mathrm{Aff}(F)$ on $F$ can be extended to a sharply $3$-transitive 
action leads to the definition of a KT-nearfield. 

\begin{definition}\rm Let $F$ be a nearfield. An involutory automorphism $\sigma$ 
of the multiplicative group of $F$ is called a {\it KT-automorphism} if 
 $(1+a^\sigma)^\sigma=1-(1+a)^\sigma$ for all $a\in F^*\setminus \{-1\}$. The pair $(F,\sigma)$ is called a {\it KT-nearfield}.
\end{definition}

\noindent
The abbreviation {\it KT} stands for {\it Karzel-Tits}.
\begin{example}\label{dickson2}\rm Let $F$ be a commutative field,  $\varphi:F^*\to \mathrm{Aut}(F)$ a coupling with $\varphi(a)=\varphi(a^{-1})$ 
for all $a\in F^*$. 
Let $a^\sigma:=a^{-1}$ (the inverse with respect to the field multiplication) for all $a\in F^*$. Then $(F^\varphi,\sigma)$ is 
a KT-nearfield. 
\end{example}

\begin{definition}
\rm Let $(F,\sigma)$ be a KT-nearfield. Set $X:=F\dot{\cup} \{\infty\}$. We embed $\mathrm{Aff}(F)$ into $\mathrm{Sym}(X)$ by 
setting $\infty g :=\infty$ for all $g\in \mathrm{Aff}(F)$. Let $\tau \in \mathrm{Sym}(X)$ be defined by
$$x\tau:= \left\{ \begin{array}{rl} 0, & \text{if } x=\infty, \\ \infty, & \text{if } x=0, \\ -x^\sigma, & \text{if } x\in F^*.
\end{array} \right.$$ We set $T_3(F):= \langle \tau, \mathrm{Aff}(F)\rangle\leq \mathrm{Sym}(X)$. 
\end{definition}

\begin{definition}\rm Let $(F,\sigma)$ be a KT-nearfield. For $a\in F^*$ let $q_a:=a^{-\sigma}\cdot a$, the {\it pseudo-square} associated to $a$. 
\end{definition}

\begin{theorem}\label{kt moufang} Let $(F,\sigma)$ be a KT-nearfield.
\begin{enumerate}
\item $T_3(F)$ is a sharply $3$-transitive permutation group on $X$ with $T_3(F)_\infty =\mathrm{Aff}(F)$ and 
$T_3(F)_{0,\infty} \cong F^*$. 
\item $\mathbb{M}(F,\tau)$ is a special Moufang set 
whose little projective group $G^\dagger$ is 
normal in $T_3(F)$. Moreover, for all $a,x\in F$ with $a\ne 0$ we have $xh_a =x\cdot q_a$. 
\end{enumerate}
\end{theorem}

\begin{proof}
\begin{enumerate}
\item This follows by \cite{KW}.
\item For every $x\in X$ let $U_x$ be a conjugate of $\mathcal{T}(F)$ contained in $T_3(F)_x$. Then
$\big( X, (U_x)_{x\in X}\big)$ is a Moufang set whose little projective group $G^\dagger$ is normal in $T_3(F)$. 
Since  $\tau$ interchanges $0$ and $\infty$, this Moufang set equals $\mathbb{M}(F,\tau)$. 
For $a\in F^\#$, we have $(-a)\tau =(-1\cdot a)\sigma =(-1)\sigma \cdot 
a\sigma =-a\sigma =-a\tau$, thus $\mathbb{M}(F,\tau)$ is special. 

\bigskip\noindent
We want to prove the formula $xh_a =x\cdot q_a$ for all $x,a\in F$ with $a\ne 0$. We have 
\begin{eqnarray*}
xh_a &=& x\tau\alpha_a\alpha_{-a\tau^{-1}}^\tau \alpha_{-(-a\tau^{-1})\tau} \\
&=& \big( (x\tau +a)\tau^{-1}-a\tau^{-1}\big)\tau +a \\
&=&-\big( a^\sigma-(a-x^\sigma)^\sigma\big)^\sigma+a \\
&=& \big( (1-x^{\sigma}\cdot a^{-1})^\sigma \cdot a^\sigma-a^\sigma \big)^\sigma +a\\
&=& \left( \big(1-(x \cdot a^{-\sigma})^\sigma \big)^\sigma-1\right)^\sigma \cdot 
a^{\sigma^2} +a \\
&=& \big( 1-(1-x\cdot a^{-\sigma})^\sigma-1\big) \cdot a+a\\
&=& (x\cdot a^{-\sigma} -1)^{\sigma^2} \cdot a +a \\
&=& x\cdot a^{-\sigma}\cdot a -a+a\\
&=& x\cdot q_a.
\end{eqnarray*}
\end{enumerate}
\end{proof}

\begin{example}\rm If $F$ is a commutative field and $\varphi$ is a coupling as in Example \ref{dickson}, then we have $\mathrm{PSL}_2(F) \leq T_3(F^\varphi) \leq \mathrm{P\Gamma L}_2(F)$. 
For example, if $F=\mathbb{F}_9$ and $\varphi$ is the unique non-trivial coupling of $F$, then $T_3(F^\varphi) =\mathrm{Mat}_{10}$, the Mathieu group of degree $10$.
\end{example}

\noindent
Recall that for $n\geq 1$ a permutation group is called {\it $n$-sharp} if the 
stabiliser of $n$ distinct points is always trivial

\begin{lemma}\label{1+2=3} Let $G$ be a $2$-sharp permutation group on a set $X$ and 
$A$ a regular subgroup of $G$. Suppose that there is $x\in X$ and $h_1,h_2,h_3 
\in H:= N_G(A)\cap G_x$ such that $a^{h_3} =a^{h_1} a^{h_2}$ for all $a\in A$. 
Then $Z(G_x) \leq H$.
\end{lemma}

\begin{proof} 
We write the group $A=(A,+,0)$ additively, identify the set $X$ with $A$, the element 
$x\in X$ with $0\in A$ and the group $A$ with the subgroup $\{\alpha_a\mid a\in A\}$ 
of $\mathrm{Sym}(A)$, where $b\alpha_a = b+a$ for all $a,b\in A$. Then 
we have $\alpha_{ah}=\alpha_a^h$ for all $a\in A$ and all $h\in H$. 
Thus $ah_3 = 0\alpha_{ah_3} =0\alpha_a^{h_3} =0\alpha_a^{h_1}\alpha_a^{h_2} =
0\alpha_{ah_1}\alpha_{ah_2}=ah_1+ah_2$ for all $a\in A$. 

\medskip\noindent
Let $t\in Z(G_0)$ and $0\ne a\in A$. Then $g_a:=\alpha_a^t\alpha_{-at}\in G_0$.
We have 
\begin{eqnarray*} g_{ah_3} \alpha_{ath_1+ath_2} 
 &=& g_{ah_3} \alpha_{ath_3} \\
&=& g_{ah_3} \alpha_{ah_3t}\\ &=& \alpha_{ah_3}^t \\
&=& \alpha_{ah_1+ah_2}^t \\
&=& \alpha_{ah_1}^t \alpha_{ah_2}^t \\
&=& g_{ah_1} \alpha_{ah_1t} g_{ah_2} \alpha_{ah_2t} \\
&=& g_{ah_1} \alpha_{ath_1} g_{ah_2}\alpha_{ath_2}.
\end{eqnarray*}
Thus 
$$ath_1 +ath_2 = 0g_{ah_3} \alpha_{ath_1+ath_2} =0g_{ah_1} \alpha_{ath_1} g_{ah_2}\alpha_{ath_2} = ath_1g_{ah_2} +ath_2$$ and hence 
$ath_1 = ath_1 g_{ah_2}$. Since $G$ is $2$-sharp, this means $g_{ah_2}=1$. 
Replacing $a$ by $ah_2^{-1}$ yields $g_a=1$. Therefore we have 
$\alpha_a^t =\alpha_{at}$ for all $a\in A$, thus $t\in H$.
\end{proof}

\begin{theorem} Let $\mathbb{M}(U,\tau)$ be a 
special Moufang set whose little projective group $G^\dagger$ is contained 
in a sharply $3$-transitive subgroup $G$ of $\mathrm{Sym}(X)$, 
where $X:=U\dot{\cup} \{\infty\}$. Then $G^\dagger \trianglelefteq G$ and 
there is a KT-nearfield $(F,\sigma)$ such that 
$(X,G)$ and $\big(F\dot{\cup} \{\infty\}, T_3(F)\big)$ are isomorphic permutation groups.
\end{theorem}

\begin{proof}
By \cite[Corollary 6.4(3)]{DST} the group $U$ is abelian.

\medskip\noindent
We claim that $U_\infty$ is normal in $G_\infty$. Note that $G_\infty$ acts 
sharply $2$-transitively on $U$. 
Let $J$ be the set of involutions of $G_\infty$ and $\overline{J}:=J\cup \{1\}$. 
If char$U=2$, then $\overline{J}\cap G_\infty$ 
acts regularly on $U$ by \cite[(3.3)]{K68}. 
Since $U_\infty \subseteq \overline{J}\cap G_\infty$ is also regularly 
on $U$, we get that 
$U_\infty =\overline{J}$ is normal in $G_\infty$.

\medskip\noindent
Now suppose that char$U=3$. Set $C:=\{x\in G\mid 
o(x)=3\}$ and $\overline{C}:=C\cup \{1\}$. 
By \cite[(2.1)]{K68} is a conjugacy class of $G$. Therefore every 
element of $C$ has exactly one fixed point. Thus $C\cap G_\infty$ is 
a conjugacy class of $G_\infty$. By \cite[(3.1)(b)]{K68} the set 
$\overline{C}\cap G_\infty$ acts regularly on $U$. Since $U_\infty \subseteq 
\overline{C} \cap G_\infty$ also acts regularly on $U$, we have $U_\infty=
\overline{C} \cap G_\infty$ and thus $U_\infty \trianglelefteq G_\infty$.

\medskip\noindent
Now suppose that char$U \ne 2,3$ and that 
there is an involution $t\in G_{0,\infty}$ normalising $U_\infty$. 
Then $t$ acts freely as an automorphism of $U$. Thus $at =-a$ for all $a\in U$. 
Now $U$ is uniquely $2$-divisible, hence we have $a =\big( \frac{1}{2}\cdot a^{-1})^t
+\frac{1}{2}\cdot a$ and hence $\alpha_a =\alpha_{-\frac{1}{2}\cdot a}^t 
\alpha_{\frac{1}{2}\cdot a} =t t^{\alpha_{\frac{1}{2}\cdot a}} \in t \overline{J}$. 
Since $t(J\cap G_\infty)$ acts regularly on $U$ by \cite[(3.1)(c)]{K68}, thus $t(J\cap G_\infty)=U_\infty$. 
Hence condition (v) and (viii) of \cite[(3.7)]{K68} apply, thus 
$U_\infty=(J\cap G_\infty)^2$ is normal in $G_\infty$ by condition (i) of the 
same theorem. Therefore we show that there is an involution $t\in G_{0,\infty}$ normalising $U_\infty$.

\medskip\noindent 
Let $\Bbbk$ be the prime field $\mathbb{M}(U,\tau)$ and 
let $e\in U^\#$ with $\mu_e=\tau$.. For alle $r\in \Bbbk^*$ we have 
$h_{r\cdot e} =r^2 \cdot \mathrm{id}_U$ by \cite{DS}[Proposition 4.6(6)]. 
Suppose that 
char$U\equiv 1$ mod $4$. Then there is an element $r\in \Bbbk^*$ with $r^2=-1$. 
Thus $h_{r\cdot e}=-\mathrm{id}$ is an involution and we are done.

\medskip \noindent
Now suppose char$U=0$ or char$U>5$. 
Since $|(\Bbbk^*)^2| \geq 3$, there are $u,v\in \Bbbk^*$ 
such that $u^2\ne \pm v^2$. Set $\ell:=u^2 -v^2$, $m:=2uv$ and $n:=u^2 +v^2$.  
For all $x\in U$ we have 
$xh_{\ell\cdot e} +xh_{m\cdot e} =\ell^2 \cdot x +m^2 \cdot x 
=(u^2 -v^2)^2 \cdot x+4u^2 v^2 \cdot x = (u^2+v^2)^2 \cdot x = 
xh_{n\cdot e}$. By \cite{K68} there is a unique involution $t$ in $G_{0,\infty}$, 
and by \ref{1+2=3} we have $t\in N_{G_\infty}(U_\infty)$. Thus $U_\infty \trianglelefteq G_\infty$ in all cases.  

\medskip\noindent 
Therefore $G$ is a sharply $3$-transitive permutation group such that 
every point stabilizer has an abelian normal subgroup acting regularly on 
the other points. Thus there is a KT-nearfield $(F,\sigma)$ such that 
$G\cong T_3(F)$ by \cite{KW}.
Since $G^\dagger$ is generated by the root 
groups, we have $G^\dagger\trianglelefteq G$. 
\end{proof}

\begin{definition}\rm Let $(F,\sigma)$ be a KT-nearfield with 
kernel $\Bbbk$. 
We set $\Bbbk_\sigma:=\Bbbk \cap (\Bbbk^*)^\sigma \cup\{0\}$.
\end{definition}

\begin{theorem}\label{ksigma} Let $(F,\sigma)$ be KT-nearfield with kernel $\Bbbk$. 
Then $\Bbbk_\sigma$ is a 
commutative subfield of $\Bbbk$ such that for all $x\in \Bbbk_\sigma$ we have 
$x^2 \in Z(F)$ and $x^\sigma=x^{-1}$ if $x\ne 0$. If 
char$F\ne 2$, then $\Bbbk=\Bbbk_\sigma$. 
\end{theorem}

\begin{proof}
This follows by the main theorem of \cite{Kerby}.
\end{proof}

We are now ready to prove Theorem $2$.

\begin{theorem} Let $(F,\sigma)$ be a KT-nearfield, $\Bbbk$ its kernel 
and $R$ its set of pseudo-squares. 
Suppose that 
\begin{itemize}
\item[\rm (i)] char$F\ne 2$ and $\dim_\Bbbk F<\infty$ or
\item[\rm (ii)] char$F=2$, $\dim_{\Bbbk_\sigma} F<\infty$ and $F$ is generated by $R$ as a $\Bbbk_\sigma$-vector space and by $C_{F^*}(R)$ as a nearring.  
\end{itemize}
Then 
$F$ is a Dickson nearfield as in Example \ref{dickson2}. 
\end{theorem}

\begin{proof} 
This is well-known if $F$ is finite, see for example \cite{W}[V, \S 5, (5.1)]. Hence we may assume that $F$ is infinite. By \ref{ksigma} $\Bbbk_\sigma$ is infinite as well.
Let $\mathbb{M}(F)$ be the Moufang set associated to $(F,\sigma)$ and $H$ its Hua subgroup. 
For $a\in \Bbbk$ we define $\lambda_a:F\to F: x\mapsto a\cdot x$. Then $\lambda_a\in \mathrm{End}_H(F)$ for all $a\in \Bbbk$ by \ref{kt moufang}. Moreover, if $a\in \Bbbk_\sigma$, then 
$q_a=a^2\in Z(F)$ by part (ii) of the main theorem of \cite{Kerby}. Therefore, 
for $x\in F$ and $a\in \Bbbk_\sigma$ we have 
\begin{eqnarray*}
xh_{b\lambda_a} &= & xh_{a\cdot b} =x\cdot q_{a\cdot b}
\\ &=&
x \cdot (a\cdot b)^{-\sigma} \cdot a\cdot b
=  x\cdot b^{-\sigma} \cdot a^2 \cdot b \\ &= & x\cdot a^2 \cdot q_b =x\lambda_a^2h_b.
\end{eqnarray*}
Therefore $\lambda_a$ is contained in the centroid. Our main theorem  \ref{main} now implies that $\mathbb{M}(F)$ is the Moufang set for a quadratic Jordan division algebra. 
By \ref{abelian} the pseudo-squares generate an abelian subgroup of $F^*$. The 
claim now follows by \cite{Kerby91}[Theorem 1.2 and Theorem 1.3].
\end{proof}

\begin{remark}\rm Suppose that $(F,\sigma)$ is KT-nearfield with char$F=2$ 
such that $\dim_{\Bbbk_\sigma} F<\infty$ 
and $F$ is generated by $R$ as 
 a $\Bbbk_\sigma^2$-vector space, where $R$ is the set of pseudo-squares in $F$ and 
 $\Bbbk_\sigma^2$ denotes the subfield 
 of squares in $\Bbbk_\sigma$. 
As in the proof of the previous theorem we may conclude that $R$ generates 
an abelian subgroup of $F^*$, hence $R \subseteq C_{F^*}(R)$.   
 By \ref{ksigma} we have $(\Bbbk_\sigma^2)^* 
 \leq Z(F) \leq C_{F^*}(R)$. Thus the nearring generated by $C_{F^*}(R)$ contains 
 the $\Bbbk_\sigma^2$-vector space generated by $R$ and hence equals $F$, therefore 
 the last condition in (ii) can be omitted in this case. In particular, this 
 holds if $\Bbbk_\sigma$ is a perfect field.   
\end{remark}

\noindent The author likes to thank Theo Grundh\"ofer for pointing out the following corollary. Note that condition (ii) is slightly weaker than 
condition (ii) in \ref{coro2} since $Z(G_{x,y}) \subseteq 
\mathrm{End}_{G_{x,y}}(U_x)$. 

\begin{coro} Let $G$ be a sharply $3$-transitive permutation group on a set $X$. Assume that 
\begin{itemize}
\item[\rm (i)] for every $x\in X$ there is a normal abelian subgroup $U_x$ of $G_x$ such that $U_x$ is regular on $X\setminus \{x\}$ and not of exponent $2$ and
\item[\rm (ii)] for $x,y\in X$ distinct, the group $U_x$ is finitely generated as a module over $\mathrm{End}_{G_{x,y}}(U_x)$.
\end{itemize}
Then there is a commutative field $F$ with char$F\ne 2$ and a group $\mathrm{PSL}_2(F) \leq \mathscr{G} \leq \mathrm{P\Gamma L}_2(F)$ such that 
$(X,G)$ and $\big( \mathrm{PG}_1(F), \mathscr{G}\big)$ are isomorphic as 
permutation groups, where $\mathrm{PG}_1(F)$ denotes the projective line over 
$F$.  
\end{coro}

\begin{proof}
We have $G=T_3(F)$ for some KT-nearfield $(F,\sigma)$. We choose an element $\infty\in X$ and identify $F$, $U_\infty$ and $X\setminus \{\infty\}$. 
The group $G_{\infty,0}$ consists of all maps $\rho_a:F\to F: x\mapsto x\cdot a$ with $a\in F^*$. Let $T\in \mathrm{End}_{G_{\infty,0}}(F)$ and set $a:=1 T$. 
We claim that $a$ is contained in the kernel of $F$ and that $xT=a\cdot x$ for all $x\in F$. Indeed, we have 
$xT= 1\rho_x T =1 T\rho_x =a\rho_x= a\cdot x$ for all $x\in F$. Moreover, for $x,y\in F$ we have $a\cdot (x+y) =(x+y)T =xT+yT =a\cdot x +a\cdot y$, hence $a$ is 
in the kernel of $F$. Conversely, if $a$ is in the kernel of $F$, then the map $x\mapsto a\cdot x$ is contained in $\mathrm{End}_{G_{\infty,0}}(F)$. 
By our assumption it follows that $F$ is finite-dimensional over $\Bbbk$. 
Therefore by the previous theorem $F$ is coupled to a commutative field 
which we also call $F$. Moreover, let $G^\dagger$ be the subgroup of $G$ 
generated by all subgroups $U_x$. Then $G^\dagger$ acts Zassenhaus transitively on 
$X$. Since the pseuso-squares generate an abelian subgroup of $F^*$, the 
two-points-stabilisers are abelian. Moreover, by \cite{DST}[Theorem 1.12] either $G^\dagger$ is 
simple or isomorphic to $\mathrm{PSL}_2(2)$ or $\mathrm{PSL}_2(3)$.  
By \ref{coro} it follows that $(X,G^\dagger)$ and $\big( \mathrm{PG}_1(F), 
\mathrm{PSL}_2(F)\big)$ are isomorphic as permutation groups. 
Since $G^\dagger$ is normal in $G$, the claim follows now by \cite{D}[IV \S 6].
\end{proof}

Note that the assumption of the corollary are satisfied if $G$ is a sharply 
$3$-transitve group of even degree $q+1$. This has already been proved in 
Zassenhaus PhD thesis, see \cite{W}[V, \S 5, (5.2)] or \cite{P}[Theorem 20.5].

\end{document}